\documentclass{article}
\usepackage{amsmath}
\usepackage{amsfonts}
\usepackage{amsthm}
\usepackage{amssymb}
\usepackage[french]{babel}
\usepackage[latin1]{inputenc}
\usepackage[T1]{fontenc}
\usepackage{hyperref}
\usepackage{eprint}
\usepackage{url}

\usepackage[active]{srcltx}

\newtheorem{thm}{Th\'eor\`eme}
\newtheorem{cor}[thm]{Corollaire}
\newtheorem{defi}[thm]{D\'efinition}

\newtheorem{lem}[thm]{Lemme}
\newtheorem{prop}[thm]{Proposition}
\newtheorem{rem}{Remarque}
\newtheorem{conj}{Conjecture}

\DeclareMathOperator{\Pic}{Pic}
\DeclareMathOperator{\BPic}{\mathbf{Pic}}

\DeclareMathOperator{\HomExt}{\mathbf{Hom-ext}}

\DeclareMathOperator{\Hom}{Hom}
\DeclareMathOperator{\Ext}{Ext}

\DeclareMathOperator{\Spec}{Spec}
\DeclareMathOperator{\BoldSpec}{\rm{Spec}}

\DeclareMathOperator{\cl}{cl}
\DeclareMathOperator{\h}{H}
\DeclareMathOperator{\CH}{CH}
\DeclareMathOperator{\Z}{Z}
\DeclareMathOperator{\Gal}{Gal}

\DeclareMathOperator{\id}{id}
\DeclareMathOperator{\Rev}{Rev}
\DeclareMathOperator{\IsCl}{IsCl}
\DeclareMathOperator{\dashab}{-Ab}
\DeclareMathOperator{\SL}{SL}
\DeclareMathOperator{\Div}{Div}
\DeclareMathOperator{\NS}{NS}
\input xy
\xyoption{all}

\begin{document}
\title{Un critère d'épointage des sections $l$-adiques}
\author{Niels Borne\footnote{Partiellement soutenu par le projet
ANR-10-JCJC 0107.} \and Michel Emsalem}
\maketitle
\bibliographystyle{alpha}

\section{Introduction}
Soit $Y$ une courbe géométriquement connexe sur un corps $k$ de type fini sur $\mathbb Q$. On dispose de la suite exacte courte des groupes fondamentaux étales 
$$( \ast _Y) \quad 1\to  \pi _1 ( Y_{\bar k} , \bar x) \to \pi _1 (Y, \bar x ) \to G_k \to 1 $$
où $ \bar k$ est une clôture algébrique, $G_k=\Gal ( \bar k /k)$ est le groupe de Galois absolu de $k$, et $ \bar x$ est un point géométrique de $Y$. Soient à présent $X$ une courbe propre lisse géométriquement irréductible, et $U$ le complémentaire dans $X$ d'un diviseur $D$ défini sur $k$. On s'intéresse dans cet article à la question suivante, dite d'épointage des sections : 
\medskip

les sections de la suite exacte $(\ast _X)$ se relèvent-elles en des sections de la suite exacte $(\ast _U)$ ? 
\medskip

Une des motivations pour l'étude de ce problème est que la conjecture des section pour $X$ de genre au moins $2$, qui affirme que toute section de $(\ast _X)$ provient d'un point $k$-rationnel de $X$, implique une réponse positive à la question de l'épointage des sections. 

 Réciproquement, il est bien connu des spécialistes qu'une réponse positive à la question de l'épointage des sections permettrait de réduire la conjecture des sections au cas de la droite projective privée d'un nombre fini de points. Pour donner un sens précis à cette assertion, on doit formuler une version de la conjecture des sections pour des courbes éventuellement non propres, où à chaque point à l'infini est associé un \og{}paquet\fg{} de (classes de conjugaison de) sections, qui a le cardinal du continu.

  Comme une référence pour cette réduction de la conjecture de sections semble faire défaut, nous en avons inclus une esquisse assez détaillée (voir \S\ref{part_epointage_prof}), en prenant le parti de remplacer les courbes affines par des orbicourbes qui leur sont \og{}homotopiquement équivalentes\fg{}, mais dont la propreté permet un traitement beaucoup plus uniforme (voir \S\ref{par_orbicourbe}), les \og{}paquets\fg{} mentionnés ci-dessus correspondant à des limites projectives de points rationnels sur ces orbicourbes.
  
  \medskip
  
  Nous nous limiterons à l'étude de la question de l'épointage des sections dans le cadre suivant. Pour tout nombre premier $l$, considérons la suite exacte 
$$( \ast _Y)^{[ab,l]} \quad 1\to  \h_1 (Y_{\bar k } , \mathbb Z_l) \to \pi _1^{[ab,l]} (Y, \bar x ) \to G_k \to 1 $$
où $\pi _1 ^{[ab,l]} (Y , \bar x)$ désigne le quotient de $\pi _1 ( Y , \bar x)$ par le noyau du morphisme naturel $\pi _1 ( Y_{\bar k } , \bar x) \to \h_1 (Y_{\bar k } , \mathbb Z_l)$ du groupe fondamental \og{}géométrique\fg{} vers son plus grand $l$-quotient abélien. Dans ce cadre se pose la question de l'épointage des sections $l$-adiques :
 
\medskip

les sections de la suite exacte $(\ast _X)^{[ab,l]}$ se relèvent-elles en des sections de la suite exacte $(\ast _U)^{[ab,l]}$ ?
\medskip

Ce type de problème a été récemment soulevé par Mohamed Saïdi (voir \cite{saidi_good_2010}). On montrera dans cet article l'énoncé suivant : 

\begin{thm}
  \label{thm_princ}
  Soient $k$ un corps de type fini sur $\mathbb Q$, $G_k$ son groupe de Galois absolu, $X/k$ une courbe propre et lisse, géométriquement connexe, $D\subset X$ un diviseur réduit, $U=X\backslash D$ l'ouvert complémentaire, et $\overline x\in U$ un point géométrique. On fixe une extension galoisienne finie $k'/k$ contenant les corps de définition de tous les points de $D$ et un nombre premier $l$ ne divisant pas $[k':k]$. 
  On note $\BPic_{X/k}$ le schéma de Picard de $X$, et $\BPic_{X,D/k}$ celui de la courbe obtenue à partir de $X$ en pinçant $D$ en un unique point rationnel (voir \S \ref{cond_suffisante_epointage}). On fait les deux hypothèses suivantes :
 \begin{enumerate}
    \item  Il existe un entier naturel $N\geq 1$ premier à $l$ et un point rationnel $p:\Spec k\rightarrow \BPic^N_{X/k}$ tel que $D$ soit inclus dans la fibre du morphisme composé $X\rightarrow \BPic^1_{X/k}\xrightarrow{\times N} \BPic^N_{X/k}$ en $p$,
    \item la classe de $\BPic^1_{X,D}$ appartient au sous-groupe $l$-divisible maximal du groupe $H^1(G_k,\BPic^0_{X,D})$. 
      \end{enumerate}
Alors toute section du morphisme naturel $\pi_1(X,\overline{x})^{[ab,l]} \rightarrow G_k $ se relève en une section du morphisme $\pi_1(U,\overline{x})^{[ab,l]}\rightarrow G_k$.
  \end{thm}

  Voici un aperçu de la preuve du Théorème \ref{thm_princ}. La deuxième condition du théorème équivaut, d'après \cite{harari-szamuely;Galois-sections-abelianized}, à l'existence d'une section du morphisme $\pi_1^{[ab,l]}(U,\overline x) \rightarrow G_k$. Le problème revient alors à montrer la surjectivité de l'application $\h^1(G_k,\h_1(U_{\overline{k}},\mathbb Z_l))\rightarrow \h^1(G_k,\h_1(X_{\overline{k}},\mathbb Z_l))$. Une condition suffisante --- beaucoup plus forte a priori que la condition d'épointage des sections $l$-adiques --- est, bien entendu, que $\h_1(U_{\overline{k}},\mathbb Z_l)$, qui est a priori une extension de $\h_1(X_{\overline{k}},\mathbb Z_l)$ par un module de Tate qu'il est facile de décrire en fonction des pointes, soit en fait une somme directe de ces deux représentations galoisiennes pures : on dira que l'homologie de $U$ est pure. Nous utilisons alors, en le généralisant un peu, un théorème d'Uwe Jannsen (voir Théorème \ref{theoreme-caracterisation-fermes-engendrant-motis-purs}), pour prouver le théorème suivant :
  
  \begin{thm} Supposons que tous les points de $D$ soient rationnels sur une extension galoisienne finie $k'$ de $k$ de degré premier à $l$. L'épimorphisme $$\h_1 ( U_ {\bar k}, \mathbb Z_l) \twoheadrightarrow \h_1(X_{\bar k } , \mathbb Z_l)$$ admet une section $G_k$-invariante si et seulement si géométriquement (c'est-à-dire sur $\overline{k}$), pour tous $z,z'$ dans le support de $D$, le diviseur $z'-z$ est de torsion première à $l$ dans la jacobienne de $X$.
  \end{thm}
  
  La condition donnée par ce théorème pour que l'homologie soit pure sera reformulée en la première condition du Théorème \ref{thm_princ}. 
  \medskip
  
  Nous rappelons dans le paragraphe \ref{section-abel-jacobi} la construction de l'application d'Abel-Jacobi, qui à une classe de cycle cohomologue à $0$ sur $X$ associe une extension de représentations galoisiennes. Le lemme clé (Lemme \ref{lemme-Jannsen}) identifie cette application à un cobord obtenu à partir de la suite de Kummer pour la jacobienne de $X$, dont on sait qu'il est injectif. Comme il semble qu'une preuve n'ait jamais été publiée, nous en avons inclus une. Celle-ci utilise un entrelacement entre la suite de Kummer et la suite de cohomologie relative qui fait écho, de manière plus complexe, à celui qu'on utilise pour définir la classe de cycle d'un diviseur en cohomologie étale.

  Enfin, dans le paragraphe \ref{exemples_homologie_pure}, nous donnons des exemples explicites de courbes modulaires épointées dont l'homologie est pure, et pour lesquelles le Théorème \ref{thm_princ} s'applique.

  \subsection{Remerciements}

  Nous remercions Damian Rössler pour nous avoir guidés vers le théorème de Manin-Mumford, et Pierre Parent pour nous avoir fourni les exemples de courbes modulaires. Nous sommes également reconnaissants envers Jakob Stix dont les commentaires éclairés ont permis d'améliorer notre texte initial. Enfin, nous remercions le rapporteur pour sa lecture attentive ainsi que pour ses remarques.

\section{La conjecture des sections}
  
\label{par_conjecture_section}

 \subsection{Énoncé}

Soit $X/k$ un schéma de type fini,
géométriquement connexe, et $\overline x:\Spec  \Omega\rightarrow X$ un point géométrique. On dispose de la \og suite exacte fondamentale\fg{}:

\begin{equation}
  \label{suite_exacte_fondamentale}
  1 \rightarrow \pi_1(X_{\overline k},\overline x)\rightarrow  \pi_1(X,\overline x) \rightarrow G_k\rightarrow 1
\end{equation}
voir \cite{grothendieck_revetements-etales-groupe-fondamental}, IX, Théorème 6.1. Par fonctorialité, un point rationnel induit une section de cette suite exacte, bien définie à conjugaison près par un élément de  $\pi_1(X_{\overline k},\overline x)$. On obtient ainsi une application~:

\[s_X: X(k) \rightarrow \HomExt_{G_k}(G_k,\pi_1(X_{\overline k},\overline x))  \]

La conjecture des sections s'énonce ainsi :

\begin{conj}[\cite{grothendieck_brief-faltings}]
  \label{conjecture_section}
Si $X$ est une courbe propre et lisse de genre supérieur ou égal à $2$ sur un corps $k$ de type fini sur $\mathbb Q$, l'application $s_X$ est bijective. \end{conj}

L'injectivité est connue et est une conséquence du théorème de  Mordell-Weil (voir par exemple \cite{stix_cuspidal}, Appendix B).

\subsection{Orbicourbes}
\label{par_orbicourbe}
Par orbifolde, on entendra un champ de Deligne-Mumford, qui est génériquement un schéma.
Une orbicourbe est une orbifolde réduite de type fini modérée sur un corps, qui est de dimension $1$. 
Les orbicourbes peuvent être construites par recollement, mais on dispose aussi d'une description agréable de leurs foncteurs des points, due à Vistoli, qu'on rappelle brièvement.

\subsubsection{Champ des racines}

On rappelle que le champ torique $[\mathbb A^1|\mathbb G_m]$ est isomorphe au champ classifiant les couples $(\mathcal L,s)$, où $\mathcal L$ est un faisceau inversible sur $X$ et $s$ est une section de ce faisceau. 

\begin{defi}[\cite{vistoli_gromov-witten}]  
  \begin{enumerate}
    \item 
  Soit un schéma $X$, muni d'un couple $(\mathcal L,s)$, où $\mathcal L$ est un faisceau inversible sur $X$ et $s$ est une section de ce faisceau. On appelle champ des racines $r$-ièmes de $(\mathcal L,s)$ le champ

$$\sqrt[r]{(\mathcal L,s)/X}=X\times_{[\mathbb A^1|\mathbb G_m]} \mathcal [\mathbb A^1|\mathbb G_m]$$

où le produit fibré est pris par rapport aux morphismes  $(\mathcal L,s): X\rightarrow \mathcal [\mathbb A^1|\mathbb G_m]$, et l'élévation à la puissance $r$ : $\cdot^{\otimes r}:\mathcal [\mathbb A^1|\mathbb G_m]\rightarrow \mathcal [\mathbb A^1|\mathbb G_m]$. 

\item Si $D$ est un diviseur de Cartier effectif sur $X$, et $ s_D$ est la section canonique de $\mathcal O_X(D)$, on note $\sqrt[r]{D/X}$ le champ $\sqrt[r]{(\mathcal O_X(D),s_D)/X}$.

\item Soit $I$ un ensemble fini, $\mathbf r=(r_i)_{i\in I}$ une famille d'entiers $r_i\geq 1$, et $\bold D=(D_i)_{i\in I}$ une famille de diviseurs de Cartier effectifs sur $X$. 
On note  $\sqrt[\bold r]{\bold D/X}$ le champ $\prod_{X,i\in I}\sqrt[r_i]{D_i/X}$. On appellera $r_i$ la multiplicité de $D_i$.

\end{enumerate}
\end{defi}

\begin{rem}
  En dimension $1$, ou plus généralement lorsque l'on considère un schéma muni d'un diviseur à croisements normaux simples, les champs des racines définis ci-dessus sont suffisants. Pour un diviseur à croisements normaux généraux, il faudrait les remplacer par les champs des racines généralisés considérés dans \cite{borne-vistoli_parabolic-log}.
\end{rem}

\subsubsection{Groupe fondamental des champs de Deligne-Mumford}

Soit $\mathcal X$ un champ de Deligne-Mumford connexe. On se fixe un point géométrique $\overline{x}$. Le groupe fondamental (étale) basé en $\overline{x}$ été défini et étudié dans \cite{noohi_fundamental-group-algebraic} et \cite{zoonekynd_van-Kampen}. Sa définition repose sur le fait que la catégorie $\Rev \mathcal X$ 
des morphismes $\mathcal Y \rightarrow \mathcal X$ d'un champ de Deligne-Mumford vers $ \mathcal X$ qui sont \emph{représentables} finis et étales est (équivalente à) une catégorie galoisienne.

Si on suppose $\mathcal X$ de type fini et géométriquement connexe sur un corps $k$, on dispose de la \og suite exacte fondamentale \fg{} \eqref{suite_exacte_fondamentale} ; la démonstration de l'exactitude est la même que dans le cas des schémas, voir \cite{grothendieck_revetements-etales-groupe-fondamental}, IX, Théorème 6.1. Comme deux points rationnels isomorphes donnent lieu à des sections conjuguées, on dispose d'une application~:

\[s_{\mathcal X}: \IsCl\mathcal X(k) \rightarrow \HomExt_{G_k}(G_k,\pi_1(\mathcal X_{\overline k},\overline x))  \]
des classes d'isomorphisme de points $k$-rationnels vers les sections de la suite exacte fondamentale à conjugaison près.

\subsubsection{Orbicourbes hyperboliques}
 
\begin{defi}\label{def-orbicourbe}
  \begin{enumerate}
	  \item  On appelle \emph{orbicourbe}\footnote{Pour l'équivalence avec la définition précédente, sur un corps algébriquement clos, voir \cite{poma_etale_2010}, Proposition 3.1. Nous pensons que l'équivalence est encore vraie sur un corps quelconque, mais nous nous contenterons d'utiliser cette dernière définition, qui est la plus pratique.}(sous-entendu par la suite : propre et lisse) sur un corps $k$ un champ des racines $\mathcal X =\sqrt[\bold r]{\bold D/X}$, où $X$ est une courbe propre et lisse sur $k$, $(D_i)_{i\in I}$ est une famille finie de diviseurs de Cartier effectifs et réduits sur $X$, $\mathbf r=(r_i)_{i\in I}$ une famille d'entiers $r_i\geq 1$, premiers à la caractéristique de $k$. Le genre de $\mathcal X$ est celui de $X$.
    \item Si $\mathcal X$ est une orbicourbe de genre $g$ sur $k$, on définit sa \emph{\og caractéristique d'Euler-Poincaré orbifolde \fg{}} par~: 
\[ \chi^{orb}(\mathcal X)= 2-2g +\sum_{i\in I} \deg D_i \frac{1-r_i}{r_i} \]
\item On dit alors que l'orbicourbe  $\mathcal X$ est \emph{hyperbolique} si $\chi^{orb}(\mathcal X)<0$.
\end{enumerate}
\end{defi}

 \begin{rem}
   \
   \begin{enumerate}
     \item $\mathcal X$ est hyperbolique si et seulement si $\mathcal X_{\overline k}$ l'est. Lorsque $k=\mathbb C$, l'orbicourbe  $\mathcal X$ est hyperbolique si et seulement si son revêtement universel est le demi-plan de Poincaré : voir \cite{noohi_uniformization}, \S 5.
      \item Si $\mathcal Y\rightarrow \mathcal X$ est un revêtement (étale) de degré $d$, alors il résulte de la formule de Hurwitz que  $\chi^{orb}(\mathcal Y)=d\chi^{orb}(\mathcal X)$, en particulier $\mathcal Y$ est hyperbolique si et seulement si $\mathcal X$ l'est.    \end{enumerate}
 \end{rem}

\begin{conj}  \label{orbi_conjecture_section}
  Si $\mathcal X$ est une orbicourbe hyperbolique géométriquement connexe sur un corps $k$ de type fini sur $\mathbb Q$, l'application $s_{\mathcal X}$ est bijective. \end{conj}

 Cette conjecture a un caractère folklorique. Elle est un peu plus fine que la conjecture des sections pour les courbes affines énoncée en termes de \og paquets \fg{} dans \cite{grothendieck_brief-faltings}. En effet, soit $X/k$ une courbe propre lisse, munie d'un diviseur $D$, telle que la courbe affine $U=X\backslash D$ soit hyperbolique, c'est-à-dire $2-2g-\deg D <0$. Alors pour $r>\frac {\deg D}{\deg D +2g-2}$, l'orbicourbe $\sqrt[r]{D/X}$ est hyperbolique. Comme d'après le lemme d'Abhyankar $\pi_1(U,\overline x)\simeq\varprojlim_r \pi_1(\sqrt[r]{D/X},\overline x)$ (voir \cite{borne_sur_2009}, Proposition 3.2.2), on en déduit que la validité de la Conjecture \ref{orbi_conjecture_section} implique que  
  \[s_U: \varprojlim_r \IsCl\sqrt[r]{D/X} (k) \rightarrow \HomExt_{G_k}(G_k,\pi_1(U_{\overline k},\overline x))  \]
  est une bijection. On a une injection canonique $U(k)\rightarrow \varprojlim_r \IsCl\sqrt[r]{D/X} (k)$, le membre de droite correspondant aux paquets auxquels il a été fait allusion dans l'introduction. Les points rationnels des champs des racines ont donc permis de compléter $U(k)$ de manière naturelle. On peut par ailleurs analyser précisément la contribution d'un point rationnel à l'infini (voir l'appendice \ref{paquets}) ce qui permet de donner une description explicite du membre de gauche, au moins lorsque $D$ est totalement décomposé.

\section{Épointage de sections}
\label{part_epointage}

\subsection{Épointage de sections profinies}
\label{part_epointage_prof}

On fixe $k$ un corps et $X/k$ une courbe propre lisse géométriquement connexe de genre $g\geq 2$.

 \begin{conj}[Conjecture d'épointage]
   \label{conjecture_epointage}
    Si $k$ est un corps de type fini sur $\mathbb Q$ et $U$ est un ouvert non vide de $X$, alors toute section de $\pi_1(X,\overline x) \rightarrow G_k$ se relève en une section de $\pi_1(U,\overline x) \rightarrow G_k$ :
    \begin{center}
\xymatrix@R=2pt{
&&&&\pi_1(U,\overline{x})\ar[dd]& \\
   &&&&&\\
  &&&&\pi_1(X,\overline{x})\ar[r]& G_k \ar@/_/[l] \ar@{.>}[uul]  }
\end{center}
\end{conj}

La Conjecture \ref{conjecture_epointage} est une conséquence de la Conjecture \ref{conjecture_section}, comme on peut le voir à partir du diagramme

\begin{center}
    \xymatrix@R=30pt@C=90pt{
    \varprojlim_r \IsCl\sqrt[r]{D/X} (k)\ar[d] \ar[r]_{s_U} & \HomExt_{G_k}(G_k,\pi_1(U_{\overline k},\overline x))\ar[d]\\
X(k)    \ar[r]_{s_X}    & \HomExt_{G_k}(G_k,\pi_1(X_{\overline k},\overline x))         
}
\end{center}
En effet, le Corollaire \ref{paquets_cor} prouvé dans l'appendice \ref{app_paquets} implique que la flèche verticale de gauche est surjective. 

Réciproquement, si la Conjecture \ref{conjecture_epointage} est vraie, alors la conjecture des sections pour $U$ (resp. pour toute orbifolde $\mathcal X$ d'espace des modules $X$) implique la conjecture des sections pour $X$. On peut se servir de ce fait pour réduire la Conjecture \ref{conjecture_section} à la Conjecture \ref{conjecture_epointage}, et à la Conjecture \ref{orbi_conjecture_section} limitée aux orbicourbes hyperboliques d'espace des modules la droite projective. En effet, en partant d'une courbe hyperbolique $X$, et en fixant un morphisme non constant arbitraire $X\rightarrow \mathbb P^1$, il est aisé de construire un morphisme étale $\mathcal X \rightarrow \mathcal P$, où $\mathcal X$ (resp. $\mathcal P$) est une orbicourbe d'espace des modules $X$ (resp. $\mathbb P^1$). Mais on peut alors montrer en utilisant l'existence et l'unicité du relèvement des chemins étales que la conjecture des sections pour $\mathcal P$ entraîne la conjecture des sections pour $\mathcal X$ (pour ce dernier point, on pourra consulter \cite{stix_evidence_2011}, Chapter 13, pour le cas des schémas).

\subsection{Abélianisation de la partie géométrique}

Pour $l$ premier, on peut considérer le plus grand quotient pro-$l$-abélien $\h_1(X_{\overline{k}},\mathbb Z_l)$ de $\pi_1(X_{\overline{k}},\overline{x})$ et pousser la suite exacte 
\[ 1 \rightarrow \pi_1(X_{\overline{k}},\overline{x}) \rightarrow \pi_1(X,\overline{x})\rightarrow G_k \rightarrow 1 \]
par $\pi_1(X_{\overline{k}},\overline{x})\twoheadrightarrow \h_1(X_{\overline{k}},\mathbb Z_l)$, et de même pour $U$ (voir figure \ref{fig:abelianisation}).
\begin{figure}
  \begin{center}
    \xymatrix@R=30pt@C=0.5pt{
    & 1 \ar[rr]& &\pi_1(U_{\overline{k}}) \ar[rr]\ar[dl]\ar[dd] &&\pi_1(U) \ar[rr]\ar[dl]\ar[dd] && G_k \ar@{=}[dl]\ar[rr]\ar@{=}[dd]  && 1 \\
   1 \ar[rr]& &\pi_1(X_{\overline{k}}) \ar[rr]\ar[dd] &&\pi_1(X) \ar[rr]\ar[dd] && G_k \ar[rr]\ar@{=}[dd] && 1 &\\
   & 1 \ar[rr]& &\h_1(U_{\overline{k}},\mathbb Z_l) \ar[rr]\ar[dl] &&\pi_1^{[ab,l]}(U) \ar[rr]\ar[dl] && G_k \ar@{=}[dl]\ar[rr] && 1 \\
   1 \ar[rr]& &\h_1(X_{\overline{k}},\mathbb Z_l) \ar[rr] &&\pi_1^{[ab,l]}(X) \ar[rr] && G_k \ar[rr] && 1 &\\
  }
  \end{center}
  \caption{Abélianisation de la partie géométrique}
  \label{fig:abelianisation}
\end{figure}
On peut donc formuler un problème d'épointage des sections $l$-adiques :

\begin{center}
\xymatrix@R=2pt{
&&&&\pi_1(U,\overline{x})^{[ab,l]}\ar[dd]& \\
   &&&&&\\
  &&&&\pi_1(X,\overline{x})^{[ab,l]}\ar[r]& G_k \ar@/_/[l] \ar@{.>}[uul]  }
\end{center}

\begin{rem}
  La conjecture d'épointage n'implique pas une réponse positive au problème d'épointage des sections $l$-adiques, mais seulement que toute section de $\pi_1^{[ab,l]}(X,\overline x) \rightarrow G_k$ \emph{venant d'une section profinie de $\pi_1(X,\overline x) \rightarrow G_k$} se relève en une section de $\pi_1^{[ab,l]}(U,\overline x) \rightarrow G_k$. Cette dernière propriété, qui est une conséquence commune de la conjecture d'épointage et d'une réponse positive au problème d'épointage des sections $l$-adiques, peut se formuler de la façon suivante : toute section profinie $G_k \to \pi_1(X,\overline{x})$ se relève au quotient intermédiaire dans la suite d'épimorphismes
  
  \[ \pi_1(U,\overline{x})\twoheadrightarrow \pi_1(X,\overline{x})\times_{\pi_1(X,\overline{x})^{[ab,l]}}\pi_1(U,\overline{x})^{[ab,l]}  \twoheadrightarrow \pi_1(X,\overline{x}).\]
  Nous comptons revenir en détail sur ce quotient intermédiaire dans un travail ultérieur.
  
\end{rem}

\subsection{Une condition suffisante d'épointage des sections $l$-adiques}

\label{cond_suffisante_epointage}
Si $\pi_1^{[ab,l]}(X,\overline x) \rightarrow G_k$ n'admet pas de section, le problème d'épointage ne se pose  pas. Dans le cas contraire, il est équivalent à l'existence d'une section pour $\pi_1^{[ab,l]}(U,\overline x) \rightarrow G_k$ et à la surjectivité de $\h^1(G_k,\h_1(U_{\overline{k}},\mathbb Z_l))\rightarrow \h^1(G_k,\h_1(X_{\overline{k}},\mathbb Z_l))$.

La première condition est complètement comprise : une condition nécessaire et suffisante à l'existence d'une section pour $\pi_1^{[ab,l]}(U,\overline x) \rightarrow G_k$ a été donnée par Harari et Szamuely :
\begin{thm}[\cite{harari-szamuely;Galois-sections-abelianized}, Remark 2.4]
  \label{harari-szamuely}
  Le morphisme $\pi_1^{[ab,l]}(U,\overline x) \rightarrow G_k$ admet une section si et seulement si la classe de $\BPic^1_{X,D}$ appartient au sous-groupe $l$-divisible maximal de $H^1(G_k,\BPic^0_{X,D})$.
\end{thm}
(Ici, $\BPic^0_{X,D}$ désigne la jacobienne généralisée de $U$, et $\BPic^1_{X,D}$ le torseur universel de degré $1$ sous $\BPic^0_{X,D}$. On rappelle que $\BPic_{X,D}=\BPic_{X_D}$, le foncteur de Picard de $X_D$, pincement de $X$ le long de $D$ au sens de \cite{ferrand_conducteur_2003}. Le degré sur $X_D$ est défini par image réciproque le long de $X\to X_D$.)

La deuxième condition est plus mystérieuse. Une condition suffisante à la surjectivité de $\h^1(G_k,\h_1(U_{\overline{k}},\mathbb Z_l))\rightarrow \h^1(G_k,\h_1(X_{\overline{k}},\mathbb Z_l))$ est l'existence d'une section  du morphisme $\h_1(U_{\overline{k}},\mathbb Z_l) \twoheadrightarrow \h_1(X_{\overline{k}},\mathbb Z_l)$, compatible avec l'action du groupe de Galois absolu $G_k$.

\subsection{Traduction en cohomologie}

On note $K=\ker(  \h_1(U_{\overline{k}},\mathbb Z_l) \twoheadrightarrow \h_1(X_{\overline{k}},\mathbb Z_l))$ ; c'est un $\mathbb Z_l$-module de type fini muni d'une action continue de $G_k$.
Comme $\mathbb Z_l$-module, $K$ est même libre de rang $r-1$ où $r=\#|X_{\overline{k}}\backslash U_{\overline{k}}|$, comme il découle du théorème donnant la structure de $\pi_1(X_{\overline{k}},\overline{x})$ (voir \cite{grothendieck_revetements-etales-groupe-fondamental}, Exposé X, Théorème 2.6)(resp. de $\pi_1(U_{\overline{k}},\overline{x})$ [voir  \cite{grothendieck_revetements-etales-groupe-fondamental}, Exposé XIII, Corollaire 2.12]). Il est plus confortable d'étudier la suite duale de cohomologie. Précisément, on a d'après \cite{grothendieck_revetements-etales-groupe-fondamental}, Exposé XI, \S 5, un isomorphisme naturel de $\mathbb Z_l$-module libres de type fini munis d'une action continue de $G_k$ :
\[ \h^1(X_{\overline{k}},\mathbb Z_l(1)) \simeq \Hom(\h_1(X_{\overline{k}},\mathbb Z_l), \mathbb Z_l(1))\]
et de même pour $U$. Le foncteur  $\Hom(\cdot, \mathbb Z_l(1))$ étant une auto-anti-équivalence de la catégorie des $\mathbb Z_l$-module libres de type fini, qui conserve les suites exactes, l'épimorphisme $\h_1(U_{\overline{k}},\mathbb Z_l) \twoheadrightarrow \h_1(X_{\overline{k}},\mathbb Z_l)$ admet une section si et seulement si le monomorphisme $\h^1(X_{\overline{k}},\mathbb Z_l(1))\hookrightarrow \h^1(U_{\overline{k}},\mathbb Z_l(1))$ admet une rétraction. Le Théorème \ref{theoreme-caracterisation-fermes-engendrant-motis-purs} dira exactement quand cette condition se produit. 

\section{Application d'Abel-Jacobi}

\label{section-abel-jacobi}
\subsection{Définition}

On reproduit ici une construction d'Uwe Jannsen (voir \cite{jannsen_mixed-motives-K-theory}). Soit $X/k$ une variété propre et lisse, $j$ un entier, $\Z^j(X)$ le groupe abélien libre engendré par les cycles de codimension $j$ sur $X$, $\Z^j(X)_0$ le sous-groupe des cycles cohomologues à zéro. On définit un morphisme :
\[ \cl': \Z^j(X)_0\rightarrow \Ext_{{\mathbb Z}_l}^1(\mathbb Z_l, \h^{2j-1}(X_{\overline{k}},\mathbb Z_l(j))) \]
de la manière suivante. Soit $z$ un cycle de codimension $j$, cohomologue à zéro, on note $Z=|z|$ son support, et $U=X\backslash Z$. L'image $\cl'(z)$ de $z$ est la classe $[E]$ de l'extension $E$ obtenue à partir de la suite de cohomologie relative par image réciproque comme dans le diagramme suivant :
\begin{center}
\xymatrix@R=2pt{
0\ar[r] &\h^{2j-1}(X_{\overline k},\mathbb Z_l(j))\ar[r]&\h^{2j-1}(U_{\overline k},\mathbb Z_l(j))\ar[r]&\h^{2j}_{Z_{\overline k}}(X_{\overline k},\mathbb Z_l(j))\ar[r]&\h^{2j}(X_{\overline k},\mathbb Z_l(j)) \\
&&&&\\
0\ar[r] &\h^{2j-1}(X_{\overline k},\mathbb Z_l(j))\ar[r]\ar@{=}[uu]& E \ar[r]\ar[uu] & \mathbb Z_l\ar[r]\ar[uu]^{\cl(z)} & 0\ar[uu]
}
\end{center}
Le groupe suivant $0$ dans la première suite exacte est nul (au moins lorsque $Z$ est lisse, ce qui suffira par la suite) car, par dualité de Poincaré (voir \cite{jannsen_mixed-motives-K-theory}, Part II, \S 6) :
\[\h^{2j-1}_{Z_{\overline k}}(X_{\overline k},\mathbb Z_l(j)) \simeq \h_{2d-2j+1}(Z_{\overline k},\mathbb Z_l(d-j))\simeq \h^{2d-2j+1}(Z_{\overline k},\mathbb Z_l(d-j))^{\vee}=0\] vu que $2d-2j+1>2\dim Z$. Le carré de droite commute car $z$ est cohomologue à zéro. Enfin, on aurait pu seulement imposer $|z|\subset Z$, l'extension $E$ est alors indépendante du choix de $Z$, par fonctorialité de la suite de cohomologie relative en la paire $(X,Z)$.

On note $\Z_Z^j(X)$ le groupe des cycles de codimension $j$ à support dans $Z$, $\Z_Z^j(X)_0$ le sous-groupe des cycles homologiquement équivalents à zéro, et de manière similaire $\h^{2j}_{Z_{\overline k}}(X_{\overline k},\mathbb Z_l(j))_0$ le noyau de $\h^{2j}_{Z_{\overline k}}(X_{\overline k},\mathbb Z_l(j))\rightarrow \h^{2j}(X_{\overline k},\mathbb Z_l(j))$. On dispose d'une extension \og universelle \fg:
\[0\ \rightarrow \h^{2j-1}(X_{\overline k},\mathbb Z_l(j))\rightarrow \h^{2j-1}(U_{\overline k},\mathbb Z_l(j))\rightarrow \h^{2j}_{Z_{\overline k}}(X_{\overline k},\mathbb Z_l(j))_0 \rightarrow 0 \]
Les éléments $\cl'(z)$, pour $z\in \Z_Z^j(X)_0$, représentent des obstructions à l'existence d'une section. Comme $\cl: \Z_Z^j(X)_0 \rightarrow \h^{2j}_{Z_{\overline k}}(X_{\overline k},\mathbb Z_l(j))$ passe au quotient par l'équivalence rationnelle, on peut définir $\cl'$ sur le sous-groupe du groupe de Chow correspondant :
\begin{defi}
  On note encore \[ \cl': \CH^j(X)_0\rightarrow \Ext_{{\mathbb Z}_l}^1(\mathbb Z_l, \h^{2j-1}(X_{\overline{k}},\mathbb Z_l(j))) \] l'application obtenue par passage au quotient, à partir de celle définie ci-dessus.
\end{defi}
\begin{rem}
    Jannsen attribue à Deligne \cite{deligne_valeurs} l'idée d'associer une extension à un cycle. Ce dernier considère d'ailleurs plutôt le torseur sous $\h^{2j-1}(X_{\overline k},\mathbb Z_l(j))$ donné par l'image réciproque de $\cl(z)$ dans $\h^{2j-1}(U_{\overline k},\mathbb Z_l(j))$. C'est bien sûr un point de vue équivalent vu l'isomorphisme \[\Ext_{{\mathbb Z}_l}^1(\mathbb Z_l, \h^{2j-1}(X_{\overline{k}},\mathbb Z_l(j)))\simeq \h^1_{cont}(G_k,\h^{2j-1}(X_{\overline k},\mathbb Z_l(j)))\;.\] où $\h^*_{cont}$ désigne la cohomologie continue \cite{jannsen_continuous}.
  \end{rem}

\subsection{Injectivité lorsque $j=1$ : preuve \og{}élémentaire\fg{} }

  \begin{thm}[\cite{jannsen_mixed-motives-K-theory},\S 9.20]
    \label{jannsen-injectivite-abel-jacobi}
    Soient $X$ une variété propre et lisse sur un corps $k$ de type fini sur $\mathbb Q$, et $l$ un nombre premier. Alors l'application d'Abel-Jacobi :
\[ \cl': \CH^1(X)_0\otimes_{\mathbb Z} \mathbb Z_l \rightarrow \Ext_{{\mathbb Z}_l}^1(\mathbb Z_l, \h^{1}(X_{\overline{k}},\mathbb Z_l(1))) \]
est injective.
\end{thm}
Pour la commodité du lecteur, nous reproduisons dans ce paragraphe le début de la preuve donnée par Uwe Jannsen, puis nous la complétons au paragraphe suivant. Pour une preuve de l'injectivité de l'application d'Abel-Jacobi dans le cadre plus général des $1$-motifs, on pourra consulter la thèse de Peter Jossen\footnote{Voir \href{http://www.jossenpeter.ch/PdfDvi/Dissertation.pdf}{\url{http://www.jossenpeter.ch/PdfDvi/Dissertation.pdf}}. Il n'est pas clair pour nous si la Proposition 6.2.1 de cette thèse donne une autre version uniquement de la première partie de la preuve, ou bien si elle donne une preuve complète du Théorème \ref{jannsen-injectivite-abel-jacobi}, qui contournerait le Lemme \ref{lemme-Jannsen}. Celui-ci nous paraît de toute façon d'un intérêt indépendant.}. 

\begin{proof}[Preuve du Théorème \ref{jannsen-injectivite-abel-jacobi}]
  Soit $A=\BPic^0_{X/k}$, et $n\geq 1$ un entier . De la suite de Kummer : \(0\rightarrow A[l^n] \rightarrow A \xrightarrow{\times l^n} A \rightarrow 0\), on tire :
  \[ 0\rightarrow \frac{A(k)}{l^n}\xrightarrow{\delta} \h^1(G_k, A(\overline{k})[l^n])\rightarrow \h^1(G_k,A)[l^n]\rightarrow 0 \]
En passant à la limite projective sur $n$, il vient :
\begin{equation}
  \label{Kummer}
0\rightarrow \varprojlim_n\frac{A(k)}{l^n}\xrightarrow{\delta} \varprojlim_n\h^1(G_k, A(\overline{k})[l^n])\rightarrow T_l(\h^1(G_k,A))\rightarrow 0
\end{equation}
En effet  $\varprojlim_n^1\frac{A(k)}{l^n}=0$, vu que les applications de transition $\frac{A(k)}{l^{n+1}}\rightarrow \frac{A(k)}{l^n}$ sont surjectives, donc le système projectif vérifie la condition de Mittag-Leffler.

Pour analyser le premier terme de \eqref{Kummer}, on note que, comme d'après le théorème de Mordell-Weil, $A(k)$ est un groupe abélien de type fini, on a \(A(k)\otimes_{\mathbb Z} \mathbb Z_l\simeq \varprojlim_n\frac{A(k)}{l^n}\), et on dispose d'un monomorphisme\footnote{Il ne s'agit pas en général d'un isomorphisme, voir par exemple \cite{stix_period-index}, \S2. La suite spectrale de Leray permet de voir que c'en est un si $X$ admet un $0$-cycle de degré $1$ défini sur $k$.} naturel $\CH^1(X)_0\simeq\Pic^0(X)\hookrightarrow \BPic^0_{X/k}(k)=A(k)$. En ce qui concerne le second terme, on a, d'après \cite{jannsen_continuous}, \S3, une suite exacte :
\[ 0\rightarrow \varprojlim_n^1 A(k)[l^n]\rightarrow \h^1_{cont}(G_k, T_l( A(\overline{k})))\rightarrow \varprojlim_n \h^1(G_k,  A(\overline{k})[l^n])\rightarrow 0\]
Or, toujours d'après le théorème de Mordell-Weil, $A(k)[l^n] $ est un $l$-groupe abélien fini, ce qui montre à nouveau que le système projectif vérifie la condition de Mittag-Leffler, et donc $ \varprojlim_n^1 A(k)[l^n]=0$. De plus, on dispose d'une suite exacte :
\[ 0 \to \Pic_0(X_{\overline{k}})[l^n] \to \Pic(X_{\overline{k}})[l^n] \to \NS(X_{\overline{k}})[l^n] \to 0 \]
où $\NS(X_{\overline{k}})$ désigne le groupe de Néron-Séveri de $X_{\overline{k}}$. Comme celui-ci est de type fini, on a, en passant à la limite projective : 
$T_l( A(\overline{k}))=T_l(\Pic_0(X_{\overline{k}})) \simeq T_l(\Pic(X_{\overline{k}}))$. La théorie de Kummer donne enfin \(T_l(\Pic(X_{\overline{k}}))\simeq \varprojlim_n \h^1(X_{\overline{k}},\mathbf{\mu}_{l^n})=\h^1(X_{\overline{k}},\mathbb Z_l(1)) \). On a donc obtenu un monomorphisme :
\[\delta: \CH^1(X)_0\otimes_{\mathbb Z} \mathbb Z_l \hookrightarrow \h^1_{cont}(G_k, \h^1(X_{\overline{k}},\mathbb Z_l(1))) \]
et le lemme suivant conclut donc la preuve du théorème :

\begin{lem}
  \label{lemme-Jannsen}
  Via l'isomorphisme naturel \(\Ext_{{\mathbb Z}_l}^1(\mathbb Z_l, \h^1(X_{\overline{k}},\mathbb Z_l(1)))\simeq \h^1_{cont}(G_k,\h^1(X_{\overline k},\mathbb Z_l(1)))\), on a \(\delta=\cl'\).
\end{lem}
\begin{proof}
  La preuve est reportée à la partie \ref{sec-preuve-lemme-jannsen}. 
\end{proof}

\end{proof}

\begin{rem}
Pour la première partie de la preuve, on pourrait également considérer la suite exacte de systèmes projectifs de faisceaux étales :
\[ 0  \rightarrow (A[l^n],l) \rightarrow (A,l) \xrightarrow{l^n} (A,\id) \rightarrow 0 \]
La suite exacte longue en cohomologie continue donne la suite exacte :
\[ 0 \rightarrow l-\Div(A(k))\rightarrow A(k)\xrightarrow{\delta} \h^1_{cont}(G_k,T_l( A(\overline{k}))
\]
où $l-\Div(A(k))$ est le sous-groupe $l$-divisible maximal de $A(k)$. Comme $A(k)$ est de type fini, c'est simplement sa torsion première à $l$, ce qui constitue donc le noyau de $\delta$. Toutefois, la preuve donnée ici fournit un peu mieux, en particulier l'isomorphisme 
\[\h^1_{cont}(G_k, T_l( A(\overline{k}))\simeq \varprojlim_n \h^1(G_k,  A(\overline{k})[l^n])\]
justifie que pour donner la preuve du Lemme \ref{lemme-Jannsen}, on peut se contenter de travailler avec des coefficients finis.
\end{rem}

\subsection{Preuve du Lemme \ref{lemme-Jannsen}}
\label{sec-preuve-lemme-jannsen}

\subsubsection{Un diagramme périodique}
\label{sub-diag-per}
On doit montrer que le diagramme suivant commute~:

\begin{center}
\xymatrix@R=14pt{
&     \BPic^0_{X/k}(k) \ar@{=}[r] &A(k)  \ar[rrr]^{\delta}    &&&    \h^1_{cont}(G_k,A[l^n](\overline{k}))\\
   &\Pic^0(X)\ar[u]   &        &  &&   \h^1_{cont}(G_k,\Pic(X_{\overline{k}})[l^n]) \ar@{=}[u]\\
  \Z^1(X)_0   \ar[r]&  \CH^1(X)_0\ar@{=}[u]   \ar[rrrr]_{\cl'}     &          & &&  \Ext^1(\mathbb Z, \h^1(X_{\overline{k}},\mathbf{\mu}_{l^n})) \ar@{=}[u]   
}
\end{center}

Les extensions en jeu se visualisent bien sur la Figure \ref{fig:period}. La donnée de $z$ dans $\Z^1(X)_0$ définit une classe $\cl(\mathcal O(z),1)$ de $\mathbb G_m$-torseur sur $X_{\overline k}$ rigidifié en dehors de $Z_{\overline k}$ par l'identité, c'est-à-dire un élément de $\h^1_{Z_{\overline k}}(X_{\overline k},\mathbb G_m)$.  Cette classe est invariante par l'action de $G_k$ et son image dans $\h^2(X_{\overline{k}},\mathbf{\mu}_{l^n})$ est nulle. On en déduit que la composée du morphisme correspondant $\mathbb Z \rightarrow \h^1_{Z_{\overline k}}(X_{\overline k},\mathbb G_m)$ avec le morphisme $ \h^1_{Z_{\overline k}}(X_{\overline k},\mathbb G_m) \to \h^2_{Z_{\overline k}}(X_{\overline k},\mu_{l^n })$ (resp. avec le morphisme $ \h^1_{Z_{\overline k}}(X_{\overline k},\mathbb G_m) \to \h^1(X_{\overline k},\mathbb G_m)$) donne par image réciproque une extension de $\mathbb Z$ par $\h^1(X_{\overline k},\mu_{l^n })$ qui par définition est $\cl'(z)$ (resp. $\delta(z)$).

\begin{figure}
  \begin{center}
\xymatrix@R=30pt{
      & \h^1_{Z_{\overline k}}(X_{\overline{k}},\mathbf{\mu}_{l^n})=0 \ar[r] & \h^1(X_{\overline{k}},\mathbf{\mu}_{l^n}) \ar[r]  & \h^1(U_{\overline{k}},\mathbf{\mu}_{l^n}) \ar[r]& \h^2_{Z_{\overline k}}(X_{\overline{k}},\mathbf{\mu}_{l^n}) \ar[r]&  \h^2(X_{\overline{k}},\mathbf{\mu}_{l^n}) \\
0\ar[r] & \h^0_{Z_{\overline k}}(X_{\overline{k}},\mathbb G_m) \ar[r] \ar[u]& \h^0(X_{\overline{k}},\mathbb G_m ) \ar[r] \ar[u] & \h^0(U_{\overline{k}},\mathbb G_m ) \ar[r] \ar[u]& \h^1_{Z_{\overline k}}(X_{\overline{k}},\mathbb G_m ) \ar[r] \ar[u]&  \h^1(X_{\overline{k}},\mathbb G_m )  \ar[u]\\
0\ar[r] & \h^0_{Z_{\overline k}}(X_{\overline{k}},\mathbb G_m) \ar[r] \ar[u]& \h^0(X_{\overline{k}},\mathbb G_m ) \ar[r] \ar[u] & \h^0(U_{\overline{k}},\mathbb G_m ) \ar[r] \ar[u]& \h^1_{Z_{\overline k}}(X_{\overline{k}},\mathbb G_m ) \ar[r] \ar[u]&  \h^1(X_{\overline{k}},\mathbb G_m )  \ar[u]\\
0\ar[r] & \h^0_{Z_{\overline k}}(X_{\overline{k}},\mathbf{\mu}_{l^n}) \ar[r] \ar[u]& \h^0(X_{\overline{k}},\mathbf{\mu}_{l^n}) \ar[r] \ar[u] & \h^0(U_{\overline{k}},\mathbf{\mu}_{l^n}) \ar[r] \ar[u]& \h^1_{Z_{\overline k}}(X_{\overline{k}},\mathbf{\mu}_{l^n}) \ar[r] \ar[u]&  \h^1(X_{\overline{k}},\mathbf{\mu}_{l^n})  \ar[u]\\
&0  \ar[u]& 0 \ar[u] &0\ar[r] \ar[u] &  \h^0_{Z_{\overline k}}(X_{\overline{k}},\mathbb G_m) \ar[r] \ar[u]& \h^0(X_{\overline{k}},\mathbb G_m ) \ar[u]_0 
}
     \end{center}
  \caption{Diagramme périodique de période $(3,-3)$}
  \label{fig:period}
\end{figure}

\subsubsection{Dans la catégorie dérivée}

On note $i:Z \rightarrow X$ l'immersion fermée, $j:U\rightarrow X$ l'immersion ouverte. La suite de cohomologie relative est incarnée par un diagramme de triangles distingués dans $D^+(\widetilde{X_{et}})$ (voir Figure \ref{fig:cohrel}) 
\begin{figure}
\begin{center}
  \xymatrix@R=30pt{
i_*Ri^! \mathbf{\mu}_{l^n}[1] \ar[r] &\mathbf{\mu}_{l^n}[1]\ar[r] &Rj_*j^*\mathbf{\mu}_{l^n}[1]\ar[r]&i_*Ri^! \mathbf{\mu}_{l^n}[2]\\
i_*Ri^! \mathbb G_m\ar[r] \ar[u]&\mathbb G_m\ar[r]\ar[u] &Rj_*j^*\mathbb G_m\ar[r]\ar[u]&i_*Ri^! \mathbb G_m[1]\ar[u]\\
i_*Ri^! \mathbb G_m\ar[r] \ar[u]&\mathbb G_m\ar[r]\ar[u] &Rj_*j^*\mathbb G_m\ar[r]\ar[u]&i_*Ri^! \mathbb G_m[1]\ar[u]\\
i_*Ri^! \mathbf{\mu}_{l^n} \ar[r] \ar[u]&\mathbf{\mu}_{l^n}\ar[r]\ar[u] &Rj_*j^*\mathbf{\mu}_{l^n}\ar[r]\ar[u]&i_*Ri^! \mathbf{\mu}_{l^n}[1]\ar[u]\\
}
\end{center}
  \caption{Cohomologie relative}
  \label{fig:cohrel}
\end{figure}dont tous les carrés commutent sauf celui en haut à droite qui anticommute. Soient $s:X\rightarrow \Spec k$ le morphisme structurel, et $G_k\dashab$ la catégorie des groupes abéliens munis d'une action continue de $G_k$.  L'identification entre $G_k\dashab$ et faisceaux de groupes abéliens sur $(\Spec k)_{et}$, et celle des objets de $G_k\dashab$ avec les complexes concentrés en degré $0$, fournit un diagramme commutatif (voir Figure \ref{fig:diagcom}).
\begin{figure}
\begin{center}
   \xymatrix@R=50pt@C=50pt{
   & & Rs_*i_*Ri^! \mathbf{\mu}_{l^n}[2]\ar[dr]&\\
   \mathbb Z\ar[r]_{\cl(\mathcal O(z),1)}\ar[rrd]_{\cl(\mathcal O(z))}\ar[urr]^{\cl(z)} &Rs_*i_*Ri^! \mathbb G_m[1]\ar[ur]\ar[dr]& & Rs_*\mathbf{\mu}_{l^n}[2]\\   
&&Rs_*\mathbb G_m[1]\ar[ur]&
}
\end{center}
\caption{Un élément de $\Hom_{D^+(G_k\dashab)}( \mathbb Z, Rs_*\mathbf{\mu}_{l^n}[2])$}
  \label{fig:diagcom}
\end{figure}
On va voir qu'aussi bien $\cl'(z)$ que $\delta(z)$ sont déterminés par l'élément $\Hom_{D^+(G_k\dashab)}( \mathbb Z, Rs_*\mathbf{\mu}_{l^n}[2])=H^2(X, \mathbf \mu_{l^n})$ figurant sur le diagramme de la Figure \ref{fig:diagcom}.

\subsubsection{La filtration de $Rs_*\mathbf{\mu}_{l^n}[2]$}
Pour un complexe de cochaînes $C^*$, et un entier relatif $a$, on note $\tau_aC^*$ la filtration croissante canonique\footnote{$\tau_aC^*$ est la \og{}bonne\fg{} troncation de $C^*$, par opposition à la troncation \og{}brutale\fg{}, voir par exemple \cite{weibel_introduction_1994}, 1.2.7, pour le cas dual d'un complexe de chaînes.} de $C^*$, définie pour un entier relatif $b$ par :
\[ 
(\tau_aC^*)^b=
\begin{cases}
  C^b & \text{si } b<a \\
  Z^a & \text{si } b=a \\
  0 & \text{si } b>a
\end{cases}
\]
On dispose d'un monomorphisme naturel $\tau_aC^*\rightarrow C^*$, on note $q_a C^*$ le complexe quotient.
Il existe en outre un morphisme naturel $\tau_a C^*\rightarrow H^a(C^*)[-a]$.

Pour en revenir à $Rs_*\mathbf{\mu}_{l^n}[2]$, on dispose donc d'une part d'un morphisme $\tau_{-1}(Rs_*\mathbf{\mu}_{l^n}[2])\rightarrow H^{-1}(Rs_*\mathbf{\mu}_{l^n}[2])[1]=H^1(X_{\overline k},\mathbb \mu_{l^n})[1]$, qui induit un morphisme $\Hom_{D^+(G_k\dashab)}( \mathbb Z, \tau_{-1}Rs_*\mathbf{\mu}_{l^n}[2])\rightarrow \Hom_{D^+(G_k\dashab)}( \mathbb Z, H^1(X_{\overline k},\mathbb \mu_{l^n})[1])=\Ext^1_{G_k\dashab}(\mathbb Z, H^1(X_{\overline k},\mathbb \mu_{l^n}))$, et d'autre part d'une suite exacte :
\[ 0 \rightarrow \Hom_{D^+(G_k\dashab)}( \mathbb Z, \tau_{-1}Rs_*\mathbf{\mu}_{l^n}[2]) \rightarrow \Hom_{D^+(G_k\dashab)}( \mathbb Z, Rs_*\mathbf{\mu}_{l^n}[2])\rightarrow \Hom_{D^+(G_k\dashab)}( \mathbb Z, q_{-1}Rs_*\mathbf{\mu}_{l^n}[2])\]
A présent, vu que 
\[
H^{i}(q_{-1}Rs_*\mathbf{\mu}_{l^n}[2])=
\begin{cases}
  0 & \text{si }i<0 \\
  H^{i}(Rs_*\mathbf{\mu}_{l^n}[2]) & \text{si }i\geq 0
\end{cases}
\]
on déduit de la suite spectrale d'hyperext
\[\Ext^p_{G_k\dashab}(\mathbb Z, \h^q(q_{-1}Rs_*\mathbf{\mu}_{l^n}[2])) \implies \Ext^{p+q}_{D^+(G_k\dashab)}(\mathbb Z, q_{-1}Rs_*\mathbf{\mu}_{l^n}[2])\]
que le morphisme 
\[\Hom_{D^+(G_k\dashab)}( \mathbb Z, q_{-1}Rs_*\mathbf{\mu}_{l^n}[2])\rightarrow \Hom_{G_k\dashab}(\mathbb Z,H^0(q_{-1}Rs_*\mathbf{\mu}_{l^n}[2]))=H^{2}(X_{\overline k},\mathbb \mu_{l^n})^{G_k}\]
est injectif. La classe $\cl(z)$ (resp. $\cl(\mathcal O(z))$) détermine un élément (resp. le même élément) de 
\[\Hom_{D^+(G_k\dashab)}( \mathbb Z, Rs_*\mathbf{\mu}_{l^n}[2])=H^2(X, \mathbf \mu_{l^n})\] qui est nul dans $H^{2}(X_{\overline k},\mathbb \mu_{l^n})$ (et donc dans $\Hom_{D^+(G_k\dashab)}( \mathbb Z, q_{-1}Rs_*\mathbf{\mu}_{l^n}[2])$) car $z$ est homologiquement équivalent à zéro par hypothèse. 
Cet élément se relève ainsi dans $\Hom_{D^+(G_k\dashab)}( \mathbb Z, \tau_{-1}Rs_*\mathbf{\mu}_{l^n}[2])$, et donne donc un élément de $\Hom_{D^+(G_k\dashab)}( \mathbb Z, H^1(X_{\overline k},\mathbb \mu_{l^n})[1])=\Ext^1_{G_k\dashab}(\mathbb Z, H^1(X_{\overline k},\mathbb \mu_{l^n}))$. \`A l'aide du lemme suivant, et de la discussion du \S \ref{sub-diag-per}, on montre qu'il s'agit de $\cl'(z)$ (resp. de $\delta(z)$), d'où l'égalité.

\begin{lem}
Dans $D^+(G_k\dashab)$ supposons donné un diagramme commutatif :
\begin{center}
	\xymatrix{
		&\mathbb Z \ar[r]\ar[d]_x & \tau_{-1} C^*\ar[d] \\
	A^*\ar[r] & B^*\ar[r] &C^*\ar[r] &A^*[1]
}
\end{center}
où la ligne du bas est un triangle exact tel que le morphisme $\h^{-1}(B^*)\to \h^{-1}(C^*)$ soit nul. Alors l'élément de $\Ext^1_{G_k\dashab}(\mathbb Z,\h^{-1}(C^*))$ défini par la composition $\mathbb Z \to \tau_{-1} C^* \to \h^{-1}(C^*)[1]$  est la classe de l'extension obtenue par produit fibré à partir du diagramme :
\begin{center}
	\xymatrix{
		&&&\mathbb Z \ar[d]_{\h^0(x)} \ar[rd]^0 &\\
	0\ar[r]& \h^{-1}(C^*)	\ar[r]&\h^0(A^*)\ar[r] & \h^0(B^*)\ar[r] &\h^0(C^*)
}
\end{center}
\end{lem}
\begin{proof}
La vérification est immédiate.	
\end{proof}

\section{Conséquence de l'injectivité en dimension $1$}

  Dans cette partie, $X$ désigne une courbe propre, lisse, géométriquement connexe sur un corps $k$ de type fini sur $\mathbb Q$. On fixe de plus $Z\subsetneq X$ un fermé strict, et on note $U=X\backslash Z$ l'ouvert complémentaire. Désormais on prendra $j=1$. On va donner une condition nécessaire et suffisante sur $Z$ pour que la suite exacte 
   \begin{equation}
   0\ \rightarrow \h^{1}(X_{\overline k},\mathbb Z_l(1))\rightarrow \h^{1}(U_{\overline k},\mathbb Z_l(1))\rightarrow \h^{2}_{Z_{\overline k}}(X_{\overline k},\mathbb Z_l(1))_0 \rightarrow 0
   \label{suite_exacte_cohomologie}
 \end{equation}
soit scindée pour presque tout nombre premier $l$.

\begin{rem}
  Comme le note Uwe Jannsen (voir \cite{jannsen_mixed-motives-K-theory}, 5.9.2, voir aussi \cite{jannsen_weights_2010}), d'après la théorie des poids en cohomologie étale, $\h^{1}(X_{\overline k},\mathbb Z_l(1))$ (resp. $\h^{2}_{Z_{\overline k}}(X_{\overline k},\mathbb Z_l(1))_0$) est une représentation pure de poids $-1$ (resp. $0$). La représentation $\h^{1}(U_{\overline k},\mathbb Z_l(1))$ représente un prototype de motif mixte dont on étudie une condition de pureté.   \end{rem}

\subsection{Cas d'un diviseur totalement décomposé}

\label{totalement-décomposé}

On suppose dans ce paragraphe que 
$Z\subsetneq X$ est un fermé strict totalement décomposé (i.e $\forall z \in |Z|\; k(z)=k$). 
 On peut alors utiliser l'additivité des $\Ext$ pour montrer que les $\cl'(z)$, pour $z\in \Z_Z^1(X)_0$, contrôlent complètement la classe de l'extension   \eqref{suite_exacte_cohomologie}.

\begin{lem}
  \label{reduction_deux_points}
  La suite exacte \eqref{suite_exacte_cohomologie} se scinde si et seulement si 
  \[\forall z \in \Z_Z^1(X)_0 \; \cl'(z)=0\; .\]
\end{lem}

\begin{rem}
  Le groupe $\Z_Z^1(X)_0$ est abélien libre de rang $r-1$, où $r=\#|Z|$, engendré par les éléments du type $z-z'$, où $z,z'$ appartiennent à $|Z|$, la condition dans le Lemme \ref{reduction_deux_points} est donc équivalente à : $\forall z,z'\in |Z|\; \cl'(z'-z)=0$. 
\end{rem}

\begin{proof}[Démonstration du Lemme \ref{reduction_deux_points}]
  On a un isomorphisme (dit de Thom-Gysin, voir \cite{deligne_poids_1975}) :
  \[\mathbb Z_l\otimes \Z_{Z_{\overline k}}^1(X_{\overline k})_0 \xrightarrow{\cl} \h^{2}_{Z_{\overline k}}(X_{\overline k},\mathbb Z_l(1))_0\]
de $\mathbb Z_l[G_k]$-modules ; par hypothèse sur $Z$ l'action de $G_k$ sur le premier membre $\mathbb Z_l\otimes \Z_{Z_{\overline k}}^1(X_{\overline k})_0\simeq \mathbb Z_l\otimes \Z_{Z}^1(X)_0$ (et donc sur le deuxième) est triviale.  Celui-ci induit à son tour un isomorphisme :
\[\Ext_{{\mathbb Z}_l}^1(\h^{2}_{Z_{\overline k}}(X_{\overline k},\mathbb Z_l(1))_0, \h^{1}(X_{\overline{k}},\mathbb Z_l(1))) \simeq \Ext_{{\mathbb Z}_l}^1(\mathbb Z_l, \h^{1}(X_{\overline{k}},\mathbb Z_l(1)))\otimes \Z_Z^1(X)_0 \]
Pour toute base $\mathcal B$ de $\Z_Z^1(X)_0$, la classe de l'extension \eqref{suite_exacte_cohomologie} est envoyée sur $\sum_{z\in \mathcal B} \cl'(z)\otimes z$.
\end{proof}

  \begin{cor}
    \label{corrolaire-totalement-decompose}
        Soit $l$ un nombre premier. La suite exacte
	\[  0\ \rightarrow \h^{1}(X_{\overline k},\mathbb Z_l(1))\rightarrow \h^{1}(U_{\overline k},\mathbb Z_l(1))\rightarrow \h^{2}_{Z_{\overline k}}(X_{\overline k},\mathbb Z_l(1))_0 \rightarrow 0\]
se scinde si et seulement si pour tout $z,z'\in |Z|$, le cycle $z'-z$ est de torsion première à $l$ dans $\CH^1(X)_0$.
\end{cor}

\begin{proof}
  Cela résulte du Lemme \ref{reduction_deux_points} et du Théorème \ref{jannsen-injectivite-abel-jacobi}.
\end{proof}

\subsection{Caractérisation des fermés donnant naissance à un motif pur}

  \subsubsection{Restriction de l'action galoisienne}

  \begin{prop}
    \label{scinde-apres-extension}
    Soit $0\rightarrow A\rightarrow B\rightarrow C\rightarrow 0$ une suite exacte de $\mathbb Z_l[G_k]$-modules continus, et $k'/k$ une extension galoisienne finie. Si $l$ est premier à $[k':k]$, alors la suite exacte se scinde si et seulement si elle se scinde comme suite exacte de $\mathbb Z_l[G_{k'}]$-modules.
  \end{prop}

  \begin{proof}
    Le sens direct est évident. Pour la réciproque, on peut soit considérer une section $G_{k'}$-invariante, et utiliser le procédé de moyenne habituel pour la rendre $G_k$-invariante (voir \cite{serre_representations-lineaires}, Chapitre 1, Théorème 1) soit invoquer la suite spectrale des $\Ext$ :
    \[ \h^p(\Gal(k'/k),\Ext_{G_{k'}}^q(C,A))\Longrightarrow\Ext_{G_k}^q(C,A)\]
    dont la suite des termes de bas degré s'écrit : 
    \[0\rightarrow  \h^1(\Gal(k'/k),\Hom_{G_{k'}}(C,A)\rightarrow \Ext_{G_k}^1(C,A) \rightarrow \Ext_{G_{k'}}^1(C,A)^{\Gal(k'/k)}\rightarrow \cdots\] 
    Or l'hypothèse sur $l$ fait que la catégorie des $\mathbb Z_l[\Gal(k'/k)]$-modules est semi-simple (c'est-à-dire toute suite exacte s'y scinde), d'où l'annulation du premier terme.
  \end{proof}

  \subsubsection{Points rationnels du foncteur de Picard}
 
  Soient $N$ un entier naturel plus grand que $1$, et $\BPic^N_{X/k}$ la composante connexe du foncteur de Picard représentant les (familles de) classes d'isomorphisme de faisceaux inversibles de degré $N$. Ce foncteur est représentable, et c'est un torseur \emph{fppf} sur $\Spec k$ de groupe structurel $\BPic^0_{X/k}$ (voir \cite{kleiman_picard-scheme},9.6.21). En fait, comme ici $X$ a des points rationnels sur des extensions finies séparables de $k$, c'est déjà un torseur pour la topologie étale sur $\Spec k$. On dispose d'un morphisme canonique $X\rightarrow \BPic^1_{X/k}$ correspondant au faisceau inversible sur $X\times X$ associé au diviseur de Cartier donné par la diagonale $\Delta:X\rightarrow X\times_k X$.

  \begin{thm}
    \label{theoreme-caracterisation-fermes-engendrant-motis-purs}
    Soient  $k'/k$  une extension galoisienne finie telle que pour tout $ z \in |Z|$, $k(z)\subset k'$, et $l$ un nombre premier ne divisant par $[k':k]$.
  La suite exacte
	\[  0\ \rightarrow \h^{1}(X_{\overline k},\mathbb Z_l(1))\rightarrow \h^{1}(U_{\overline k},\mathbb Z_l(1))\rightarrow \h^{2}_{Z_{\overline k}}(X_{\overline k},\mathbb Z_l(1))_0 \rightarrow 0\]
se scinde si et seulement s'il existe un entier naturel $N\geq 1$ premier à $l$ et un point rationnel $p:\Spec k\rightarrow \BPic^N_{X/k}$ tel que $Z=X\backslash U$ soit inclus dans la fibre de $X\rightarrow \BPic^1_{X/k}\xrightarrow{\times N} \BPic^N_{X/k}$ en $p$.  \end{thm}

\begin{proof}
  Cela résulte du Corollaire \ref{corrolaire-totalement-decompose}, de la Proposition \ref{scinde-apres-extension}, et du lemme suivant :
  \begin{lem}
    \label{lemme-point-rationnel-PicN}
    Les notations sont celles du Théorème \ref{theoreme-caracterisation-fermes-engendrant-motis-purs}.
   Soit $N\geq 1$ un entier naturel. Le fermé $Z=X\backslash U$ est inclus dans une fibre de $X\rightarrow \BPic^1_{X/k}\rightarrow \BPic^N_{X/k}$ en un point rationnel $p:\Spec k\rightarrow \BPic^N_{X/k}$ si et seulement si
   l'image de $Z(k')$ dans $\BPic^1_{X/k}(k')$ est incluse dans une orbite sous $\BPic^0_{X/k}[N](k')$.

  \end{lem}
  \begin{proof}
    Ceci découle essentiellement du fait que le morphisme $\BPic^1_{X/k}\rightarrow \BPic^N_{X/k}$ est un $\BPic^0_{X/k}[N]$-torseur.
    Pour le sens direct : notons $\pi$ ce torseur,
il induit un $\BPic^0_{X/k}[N]$-torseur $\pi^{-1}(p)\rightarrow \Spec k$, et on a donc un isomorphisme naturel :
    \[ \BPic^0_{X/k}[N]\times_{\Spec k} \pi^{-1}(p) \simeq \pi^{-1}(p) \times_{\Spec k} \pi^{-1}(p) \]
    Si $k'/k$ est une extension arbitraire, un couple $(z,z')$ d'éléments de $Z(k')$, donc de $\pi^{-1}(p)(k')$ par hypothèse, définit unique élément $\alpha$ de  $\BPic^0_{X/k}[N](k')$ tel que $z'-z=\alpha$. Réciproquement, on dispose d'une application
    \[  Z(k')\rightarrow X(k') \rightarrow \BPic^1_{X/k} (k') \rightarrow \frac{\BPic^1_{X/k}(k') }{\BPic^0_{X/k}[N](k') }\rightarrow \frac{\BPic^1_{X/k}}{\BPic^0_{X/k}[N]}\left(k'\right)\]
    dont l'image est par hypothèse réduite à un point $p: \Spec k' \rightarrow \BPic^1_{X/k}/\BPic^0_{X/k}[N]$. Comme l'application en question est $\Gal(k'/k)$-équivariante, ceci montre que $k(p)=k$, et on vérifie alors aisément que $Z\rightarrow X \rightarrow \BPic^N_{X/k}$ se factorise par la fibre de $\pi$ en le $k$-point induit par $p$.
     \end{proof}

\end{proof}

\subsection{Fin de la preuve du Théorème \ref{thm_princ}}

Le Théorème \ref{thm_princ} est une conséquence de la condition suffisante d'épointage des sections $l$-adiques donnée \S \ref{cond_suffisante_epointage} et du Théorème \ref{theoreme-caracterisation-fermes-engendrant-motis-purs}.

    \subsection{Exemples de courbes dont l'homologie est pure}
    \label{exemples_homologie_pure}

    \subsubsection{Période des courbes}

    La condition donnée dans le Théorème \ref{theoreme-caracterisation-fermes-engendrant-motis-purs} implique que $[\BPic^N_{X/k}]=0$ dans $\h^1(G_k,\BPic^0_{X/k})$. D'autre part on a un isomorphisme de $\BPic^0_{X/k}$-torseurs $\BPic^N_{X/k}\simeq \left(\BPic^1_{X/k}\right)^N$. On sait par ailleurs que $[\BPic^1_{X/k}]$ est d'ordre fini dans $\h^1(G_k,\BPic^0_{X/k})$, égal à la \emph{période} de la courbe (cardinal du conoyau $\deg:\CH_0(X_{\overline k})^{G_k} \rightarrow \mathbb Z$), voir \cite{Eriksson-Scharaschkin,stix_period-index}. La condition du Théorème \ref{theoreme-caracterisation-fermes-engendrant-motis-purs} implique que $N$ est un multiple de la période de la courbe.
    \subsubsection{Intersection d'une courbe avec sa jacobienne}
 Il est a priori difficile de construire des exemples explicites, comme le montre le théorème suivant :
\begin{thm}[Raynaud 1983, conjecture de Manin-Mumford]
Soit $k$ un corps de nombres, $X/k$ une courbe propre lisse de genre $g\geq2$.  
    On fixe un $k$-plongement $X\hookrightarrow \BPic^0_{X/k}$. 
    Alors l'ensemble $X(\overline{k})\cap \BPic^0_{X/k}(\overline{k})_{tors}$ est fini.
  \end{thm}
  Il en résulte que si le $k$-plongement $X\hookrightarrow \BPic^0_{X/k}$ est associé à un point $z \in Z(k)$, l'ensemble $Z(\overline{k})$ devra, pour satisfaire à l'hypothèse du Théorème \ref{theoreme-caracterisation-fermes-engendrant-motis-purs}, être contenu dans l'ensemble fini $X(\overline{k})\cap \BPic^0_{X/k}(\overline{k})_{tors}$.
  
   \subsubsection{Courbes modulaires}

   Cependant les courbes modulaires sont des exemples naturels, au vu du théorème : 
\begin{thm}[Manin-Drinfeld 1972]
  Soit $\Gamma \subset \SL(2,\mathbb Z)$ un sous-groupe de congruence, $X_\Gamma$ la courbe modulaire associée. 
  La classe d'un diviseur de degré $0$ \emph{à support dans les pointes} est un élément de torsion de $\Pic^0(X_\Gamma)$. 
\end{thm}

On peut considérer par exemple $X_1(p)$, avec  $p\geq 3$ premier,  qui est définie sur $\mathbb Q$. Il y a exactement $\frac{p-1}{2}$ pointes rationnelles et une pointe définie sur $\mathbb Q(\zeta_p)^+$, le sous-corps réel maximal du corps cyclotomique $\mathbb Q(\zeta_p)$.

Un autre exemple est donné par $X_{split}(p)$, avec  $p\geq 3$ premier, qui est également définie sur $\mathbb Q$, et a une seule pointe rationnelle  et $\frac{p-1}{2}$ pointes définies sur $\mathbb Q(\zeta_p)^+$ (voir \cite{Bilu-Parent_Serre}).

Dans les deux cas, si l'on prend pour $U$ l'ouvert  complémentaire  des pointes, le Théorème \ref{theoreme-caracterisation-fermes-engendrant-motis-purs} s'applique, et donc $\h_1(U_{\overline{k}},\mathbb Z_l)$ est \og{}pure\fg{} pour presque tout nombre premier $l$. Dans le premier cas au moins, on peut préciser l'ensemble des $l$ convenables, car le cardinal du groupe des classes d'idéaux à support dans les pointes a été calculé pour $X_1(p)$ (voir \cite{Takagi_cuspidal-class-number-formula}). Nous ignorons si c'est aussi le cas pour $X_{split}(p)$. La présence de pointes non rationnelles fait que le théorème originel de Jannsen (voir \cite{jannsen_mixed-motives-K-theory}) ne s'applique pas. Toutefois si l'on se restreint aux courbes modulaires, la démarche adoptée ici est un peu artificielle, car elle se fonde sur le théorème de Manin-Drinfeld, dont une démonstration repose précisément sur la pureté de $\h_1(U_{\mathbb C},\mathbb Q)$ comme structure de Hodge (voir \cite{Elkik_Manin-Drinfeld}). Il semble donc fort probable qu'une démonstration directe de la pureté de $\h_1(U_{\overline{k}},\mathbb Z_l)$ comme représentation $l$-adique existe.

Pour en revenir au problème d'épointage des sections, on peut montrer facilement dans ces deux exemples que $\pi_1^{[ab,l]}(U,\overline x) \rightarrow G_k$ admet une section. On n'a pas besoin du Théorème \ref{harari-szamuely}, mais la présence d'une pointe rationnelle permet de montrer directement que $\pi_1(U,\overline x) \rightarrow G_k$ admet une section (voir par exemple \cite{stix_cuspidal}, Proposition 8). On conclut donc pour ces deux exemples que la seconde condition du Théorème \ref{thm_princ} est également vérifiée, et donc que la conclusion du théorème est vraie pour ces courbes.

\appendix
\section{Paquets de points rationnels des champs des racines}
\label{app_paquets}

\label{paquets}
\begin{prop}
   \label{paquets_prop}
  Soit $X$ un schéma, $D$ un diviseur de Cartier effectif et réduit, $r$ un entier, $\sqrt[r]{D/X}=\mathcal X$ le champ des racines associé, $\pi :\mathcal X\rightarrow X$ le morphisme canonique vers l'espace des modules. Alors toute trivialisation du fibré normal $\mathcal N_D\simeq \mathcal O_D$ induit une bijection entre classes d'isomorphisme de sections de $\pi_D:\pi^{-1}D\rightarrow D$ et $H^1(D,\boldsymbol{\mu}_r)$.     
\end{prop}

\begin{proof}
 
 \ 
 Le choix de $\mathcal N_D\simeq \mathcal O_D$ permet de rendre $2$-cartésienne la face du bas du diagramme 
\begin{center}
\xymatrix@R=2pt{
&&\mathcal X \ar[rr]\ar[dddd] && [\mathbb A^1/\mathbb G_m]\ar[dddd]\\
&&&&\\
&\pi^{-1}D\ar[rr]\ar[dddd]\ar[ruu]&& \left[\Spec \frac{\mathbb Z[x,y]}{y^r-x}/ \boldsymbol{\mu}_r\right] \ar[dddd]\ar[ruu]&\\
&&&&\\
&& X \ar[rr] &&[\mathbb A^1/\mathbb G_m]\\
&&&&\\
&D\ar[rr]\ar[ruu]&& \mathbb A^1\ar[ruu]&\\
&&&&&&&\\
}
\end{center}
où le morphisme $D\rightarrow \mathbb A^1$ correspond à $0\in H^0(D,\mathcal O_D)$. On en déduit un isomorphisme de $D$-champs : $\pi^{-1}D\simeq \left[\BoldSpec (\mathcal O_D[y]/y^r) /\boldsymbol{\mu}_r\right] $. Or $D$ étant réduit, toute section se factorise de manière unique par $\left[\BoldSpec (\mathcal O_D[y]/y^r) /\boldsymbol{\mu}_r\right]_{red}\simeq B_D  \boldsymbol{\mu}_r$.

\end{proof}

\begin{cor}
  \label{paquets_cor}
  Soit $\mathcal X =\sqrt[\bold r]{\bold D/X}$ une orbicourbe sur un corps $k$, où $X$ est une courbe sur $k$, et $(D_i)_{i\in I}$ est une famille finie de diviseurs de Cartier effectifs et réduits sur $X$. Alors le morphisme $\mathcal X \rightarrow X$ induit un épimorphisme sur les (classes d'isomorphismes) de points rationnels. Si $x\in X(k)$ est dans le support de $D_i$, le choix d'un vecteur tangent en $x$ induit une bijection entre classes d'isomorphismes de $\pi^{-1}(x)(k)$ et $k^*/{k^{*}}^r$.

\end{cor}

\begin{proof}
  Pour une courbe, le fibré normal s'identifie au fibré tangent, et est trivial. Le Corollaire \ref{paquets_cor}  résulte donc immédiatement de la Proposition \ref{paquets_prop} et de la théorie de Kummer.
\end{proof}

\bibliography{biblio_section}
\end{document}